\documentclass{birkjour}
\usepackage{hyperref,color}

 \newtheorem{theorem}{Theorem}[section]
 \newtheorem{corollary}[theorem]{Corollary}
 \newtheorem{lemma}[theorem]{Lemma}
 \numberwithin{equation}{section}
\newcommand{\cB}{\mathcal{B}}
\newcommand{\cP}{\mathcal{P}}
\newcommand{\C}{\mathbb{C}}
\newcommand{\N}{\mathbb{N}}
\newcommand{\T}{\mathbb{T}}
\newcommand{\R}{\mathbb{R}}
\newcommand{\Z}{\mathbb{Z}}
\begin{document}
\title[Density of Analytic Polynomials in Abstract Hardy Spaces]
{Density of Analytic Polynomials\\ in Abstract Hardy Spaces}
\author{Alexei Yu. Karlovich}
\address{%
Centro de Matem\'atica e Aplica\c{c}\~oes,\\
Departamento de Matem\'a\-tica, \\
Faculdade de Ci\^encias e Tecnologia,\\
Universidade Nova de Lisboa,\\
Quinta da Torre, \\
2829--516 Caparica, Portugal}
\email{oyk@fct.unl.pt}
\thanks{%
This work was supported by the Funda\c{c}\~ao para a Ci\^encia e a
Tecnologia (Portu\-guese Foundation for Science and Technology)
through the project
UID/MAT/00297/2013 (Centro de Matem\'atica e Aplica\c{c}\~oes).}
\begin{abstract}
Let $X$ be a separable Banach function space on the unit circle $\T$
and $H[X]$ be the abstract Hardy space built upon $X$. We show that
the set of analytic polynomials is dense in $H[X]$ if the Hardy-Littlewood
maximal operator is bounded on the associate space $X'$. This result
is specified to the case of variable Lebesgue spaces.
\end{abstract}


\maketitle

\section{Introduction}
For $1\le p\le\infty$, let $L^p:=L^p(\T)$ be the Lebesgue space on the unit
circle $\T:=\{z\in\C:|z|=1\}$ in the complex plane $\C$. For $f\in L^1$, let
\[
\widehat{f}(n):=\frac{1}{2\pi}
\int_{-\pi}^\pi f(e^{i\varphi})e^{-in\varphi}\,d\varphi,
\quad n\in\Z,
\]
be the sequence of the Fourier coefficients of $f$.
The classical Hardy spaces $H^p$ are given by
\[
H^p:=\big\{f\in L^p\ :\ \widehat{f}(n)=0\quad\mbox{for all}\quad n<0\big\}.
\]
A function of the form
\[
q(t)=\sum_{k=0}^n\alpha_k t^k,
\quad
t\in\T,
\quad
\alpha_0,\dots,\alpha_n\in\C,
\]
is said to be an analytic polynomial on $\T$. The set of all analytic
polynomials is denoted by $\cP_A$. It is well known that that the set
$\cP_A$ is dense in $H^p$ whenever $1\le p<\infty$
(see, e.g., \cite[Chap.~III, Corollary~1.7(a)]{C91}).

Let $X$ be a Banach space continuously embedded in $L^1$.
Following \cite[p.~877]{Xu92}, we will consider the abstract Hardy space 
$H[X]$ built upon the space $X$, which is defined by
\[
H[X]:=\big\{f\in X:\ \widehat{f}(n)=0\quad\mbox{for all}\quad n<0\big\}.
\]
It is clear that if $1\le p\le\infty$, then $H[L^p]$ is the classical Hardy
space $H^p$. The aim of this note is to find sufficient conditions for
the density of the set $\cP_A$ in the space $H[X]$ when $X$ falls into
the class of so-called Banach function spaces.

We equip $\T$ with the normalized Lebesgue measure $dm(t)=|dt|/(2\pi)$.
Let $L^0$ be the space of all measurable complex-valued functions on $\T$. 
As usual, we do not distinguish functions, which are equal almost everywhere
(for the latter we use the standard abbreviation a.e.).  Let $L^0_+$ be the
subset of functions in $L^0$ whose values lie in $[0,\infty]$. The
characteristic function of a measurable set $E\subset\T$ is denoted by 
$\chi_E$.

Following \cite[Chap.~1, Definition~1.1]{BS88}, a mapping 
$\rho: L_+^0\to [0,\infty]$ is called a Banach function norm
if, for all functions $f,g, f_n\in L_+^0$ with $n\in\N$, for all
constants $a\ge 0$, and for all measurable subsets $E$ of $\T$, the
following  properties hold:
\begin{eqnarray*}
{\rm (A1)} & &
\rho(f)=0  \Leftrightarrow  f=0\ \mbox{a.e.},
\
\rho(af)=a\rho(f),
\
\rho(f+g) \le \rho(f)+\rho(g),\\
{\rm (A2)} & &0\le g \le f \ \mu-\mbox{a.e.} \ \Rightarrow \ 
\rho(g) \le \rho(f)
\quad\mbox{(the lattice property)},\\
{\rm (A3)} & &0\le f_n \uparrow f \ \mbox{a.e.} \ \Rightarrow \
       \rho(f_n) \uparrow \rho(f)\quad\mbox{(the Fatou property)},\\
{\rm (A4)} & & m(E)<\infty\ \Rightarrow\ \rho(\chi_E) <\infty,\\
{\rm (A5)} & &\int_E f(t)\,dm(t) \le C_E\rho(f)
\end{eqnarray*}
with the constant $C_E \in (0,\infty)$ that may depend on $E$ and 
$\rho$,  but is independent of $f$. When functions differing only on 
a set of measure  zero are identified, the set $X$ of all functions 
$f\in L^0$ for  which  $\rho(|f|)<\infty$ is called a Banach function
space. For each $f\in X$, the norm of $f$ is defined by
$\|f\|_X :=\rho(|f|)$.
The set $X$ under the natural linear space operations and under 
this norm becomes a Banach space (see 
\cite[Chap.~1, Theorems~1.4 and~1.6]{BS88}). 
If $\rho$ is a Banach function norm, its associate norm 
$\rho'$ is defined on $L_+^0$ by
\[
\rho'(g):=\sup\left\{
\int_\T f(t)g(t)\,d\mu(t) \ : \ 
f\in L_+^0, \ \rho(f) \le 1
\right\}, \ g\in L_+^0.
\]
It is a Banach function norm itself \cite[Chap.~1, Theorem~2.2]{BS88}.
The Banach function space $X'$ determined by the Banach function norm
$\rho'$ is called the associate space (K\"othe dual) of $X$. 
The associate space $X'$ can be viewed a subspace of the (Banach) 
dual space $X^*$. 

The distribution function $m_f$ of an a.e. finite function $f\in L^0$ 
is defined by
\[
m_f(\lambda):=m\{t\in \T:|f(t)|>\lambda\},\quad\lambda\ge 0.
\]
Two a.e. finite functions $f,g\in L^0$ are said to be equimeasurable if 
\[
m_f(\lambda)=m_g(\lambda) 
\quad\mbox{for all}\quad \lambda\ge 0.
\] 
The non-increasing rearrangement of an a.e. finite function 
$f\in L^0$ is defined by 
\[
f^*(x):=\inf\{\lambda: m_f(\lambda)\le x\},\quad x\ge 0.
\]
We refer to \cite[Chap.~2, Section~1]{BS88} and 
\cite[Chap.~II, Section 2]{KPS82} for properties of distribution functions 
and non-increasing rearrangements. A Banach function space $X$ is called 
re\-ar\-range\-ment-invariant if for every pair of a.e. finite 
equimeasurable functions $f,g \in L^0$, one has the following property: 
if $f\in X$, then $g\in X$ and the equality  $\|f\|_{X}=\|g\|_{X}$ holds. 
Lebesgue spaces $L^p$, $1\le p\le\infty$, as well as, more general
Orlicz spaces, Lorentz spaces, and Marcinkiewicz spaces are classical 
examples of rearrangement-invariant Banach function spaces (see
\cite{BS88,KPS82}). For more recent examples of rearrangement-invariant 
spaces, like Ces\`aro, Copson, and Tandori spaces, we refer to the
paper of Maligranda and Le\'snik \cite{LM16}.

One of our motivations for this work is the recent progress in the study
of Harmonic Analysis in the setting of variable Lebesgue spaces
\cite{CF13,DHHR11,KMRS16}. Let $\mathfrak{P}(\T)$ be the set
of all measurable functions $p: \T\to[1,\infty]$. For 
$p\in\mathfrak{P}(\T)$, put
\[
\T_\infty^{p(\cdot)} :=\{t\in \T\ :\  p(t)=\infty\}.
\]
For a measurable function $f: \T\to\C$, consider
\[
\varrho_{p(\cdot)}(f)
:=
\int_{\T\setminus \T_\infty^{p(\cdot)}}|f(t)|^{p(t)}dm(t)
+\|f\|_{L^\infty(\T_\infty^{p(\cdot)})}.
\]
According to \cite[Definition~2.9]{CF13}, the variable Lebesgue space
$L^{p(\cdot)}$ is defined as the set of all measurable functions
$f:\T\to\C$ such that $\varrho_{p(\cdot)}(f/\lambda)<\infty$
for some $\lambda>0$. This space is a Banach function space with respect
to the Luxemburg-Nakano norm given by
\[
\|f\|_{L^{p(\cdot)}}:=\inf\{\lambda>0: 
\varrho_{p(\cdot)}(f/\lambda)\le 1\}
\]
(see, e.g., \cite[Theorems~2.17, 2.71 and Section~2.10.3]{CF13}). 
If $p\in\mathfrak{P}(\T)$ is constant, then $L^{p(\cdot)}$ is 
nothing but the standard Lebesgue space $L^p$. If $\in\mathfrak{P}(\T)$
is not constant, then $L^{p(\cdot)}$ is not rearrangement-invariant
\cite[Example~3.14]{CF13}. Variable Lebesgue spaces are often called 
Nakano spaces. We refer to Maligranda's paper \cite{M11} for the role 
of Hidegoro Nakano in the study of variable Lebesgue spaces. 
The associate space of $L^{p(\cdot)}$ is isomorphic to the space 
$L^{p'(\cdot)}$, where $p'\in\mathfrak{P}(\T)$ is defined so that 
$1/p(t)+1/p'(t)=1$ for a.e. $t\in\T$ with the usual convention 
$1/\infty:=0$ \cite[Theorem~3.2.13]{DHHR11}. For 
$p\in\mathfrak{P}(\T)$, put
\[
p_-:=\operatornamewithlimits{ess\,inf}_{t\in\T}p(t),
\quad
p_+:=\operatornamewithlimits{ess\,sup}_{t\in\T}p(t).
\]
The space variable Lebesgue space $L^{p(\cdot)}$ is separable
if and only if $p_+<\infty$ (see, e.g., \cite[Theorem~2.78]{CF13}).

The following result is a kind of folklore.
\begin{theorem}\label{th:density-analytic-polynomials-RI}
Let $X$ be a separable rearrangement-invariant Banach function space on $\T$.
Then the set of analytic polynomials $\cP_A$ is dense in the abstract Hardy
space $H[X]$. Moreover, for every $f\in H[X]$, there is a sequence of analytic
polynomials $\{p_n\}$ such that $\|p_n\|_X\le\|f\|_X$ for all $n\in\N$ and
$p_n\to f$ in the norm of $X$ as $n\to\infty$.
\end{theorem}
Surprisingly enough, we could not find in the literature neither
Theorem~\ref{th:density-analytic-polynomials-RI} explicitly 
stated nor any result on the density of $\cP_A$ in abstract Hardy spaces 
$H[X]$ in the case when $X$ is an arbitrary Banach function space beyond 
the class of rearrangement-invariant spaces. The aim of this note is to fill
in this gap.

Given $f\in L^1$, the Hardy-Littlewood maximal function is defined by
\[
(Mf)(t):=\sup_{I\ni t}\frac{1}{m(I)}\int_I|f(\tau)|\,dm(\tau),
\quad t\in\T,
\]
where the supremum is taken over all arcs  $I\subset\T$ containing $t\in\T$.
The operator $f\mapsto Mf$ is called the Hardy-Littlewood maximal operator.
\begin{theorem}[Main result]
\label{th:density-analytic-polynomials-BFS}
Suppose $X$ is a separable Banach function space on $\T$. If the 
Hardy-Littlewood maximal operator $M$ is bounded on the associate space 
$X'$, then the set of analytic polynomials  $\cP_A$ is dense in the 
abstract Hardy space $H[X]$.
\end{theorem} 
To illustrate this result in the case of variable Lebesgue spaces,
we will need the following classes of variable exponents.
Following \cite[Definition~2.2]{CF13}, one says that $r:\T\to\R$ is locally 
log-H\"older continuous if there exists a constant $C_0>0$ such that
\[
|r(x)-r(y)|=C_0/(-\log|x-y|))\quad\mbox{for all}\quad x,y\in\T,\quad |x-y|<1/2.
\]
The class of all locally log-H\"older continuous functions is denoted by
$LH_0(\T)$. If $p_+<\infty$, then $p\in LH_0(\T)$ if and only if 
$1/p\in LH_0(\T)$. By \cite[Theorem~4.7]{CF13}, if $p\in\mathfrak{P}(\T)$
is such that $1<p_-$ and $1/p\in LH_0(\T)$, then the Hardy-Littlewood maximal
operator $M$ is bounded on $L^{p(\cdot)}$. This condition was initially
referred to as ``almost necessary" (see \cite[Section~4.6.1]{CF13}
for further details). However, Lerner \cite{L05} constructed an example
of discontinuous variable exponent such that the Hardy-Littlewood maximal 
operator is bounded on $L^{p(\cdot)}$. 

Kapanadze and Kopaliani \cite{KK08} developed further Lerner's ideas. They
considered the following class of variable exponents. Recall
that a function $f\in L^1$ belongs to the space $BMO$ if
\[
\|f\|_*:=\sup_{I\subset\T}\frac{1}{m(I)}\int_I|f(t)-f_I|\,dm(t)<\infty,
\]
where $f_I$ is the integral average of $f$ on the arc $I$ and the supremum
is taken over all arcs $I\subset\T$. For $f\in BMO$, put
\[
\gamma(f,r):=\sup_{m(I)\le r}\frac{1}{m(I)}\int_I|f(t)-f_I|\,dm(t).
\]
Let $VMO^{1/|\log|}$ be the set of functions $f\in BMO$
such that 
\[
\gamma(f,r)=o(1/|\log r|)\quad\mbox{ as }\quad r\to 0.
\]
Note that $VMO^{1/|\log|}$ contains discontinuous functions.
We will say that $p\in\mathfrak{P}(\T)$ belongs to the Kapanadze-Kopaliani
class $\mathfrak{K}(\T)$ if $1<p_-\le p_+<\infty$ and $p\in VMO^{1/|\log|}$.
It is shown in \cite[Theorem~2.1]{KK08} that if $p\in\mathfrak{K}(\T)$, 
then the Hardy-Littelwood maximal operator $M$ is bounded on the variable
Lebesgue space $L^{p(\cdot)}$. 
\begin{corollary}
Suppose $p\in\mathfrak{P}(\T)$. If $p_+<\infty$ and $p\in LH_0(\T)$
or if $p'\in\mathfrak{K}(\T)$, then the set of analytic polynomials
$\cP_A$ is dense in the abstract Hardy space $H[L^{p(\cdot)}]$
built upon the variable Lebesgue space $L^{p(\cdot)}$.
\end{corollary}
The paper is organized as follows. In Section~\ref{sec:preliminaries},
we prove that the separability of a Banach function space $X$ is
equivalent to the density of the set of trigonometric polynomials $\cP$
in $X$ and to the density of the set of all continuous functions $C$ in $X$.
Further, we recall a pointwise estimate of the Fej\'er means $f*K_n$,
where $K_n$ is the $n$-th Fej\'er kernel, by the Hardy-Littlewood maximal function $Mf$. In Section~\ref{sec:proofs} we show that
the norms of the operators $F_nf=f*K_n$ are uniformly bounded on a
Banach function space $X$ if $X$ is rearrangement-invariant or if the 
Hardy-Littlewood maximal operator is bounded on $X'$. 
Moreover, if $X$ is rearrangement-invariant, then $\|F_n\|_{\cB(X)}\le 1$
for all $n\in\N$. Further, we prove that under the assumptions of 
Theorem~\ref{th:density-analytic-polynomials-RI}
or~\ref{th:density-analytic-polynomials-BFS}, $\|f*K_n-f\|_X\to 0$
as $n\to\infty$. It remains to observe that $f*K_n\in\cP_A$ if $f\in H[X]$,
which will complete the proof of 
Theorems~\ref{th:density-analytic-polynomials-RI}
and~\ref{th:density-analytic-polynomials-BFS}.
\section{Preliminaries}\label{sec:preliminaries}
\subsection{Elementary lemma}
We start with the following elementary lemma, whose proof can be found, e.g.,
in \cite[Chap. III, Proposition~1.6(a)]{C91}.
Here and in what follows, the space of all bounded linear operators on a 
Banach space $E$ will be  denoted by $\cB(E)$.
\begin{lemma}\label{le:elementary}
Le $E$ be a Banach space and $\{T_n\}$ be a sequence of bounded operators 
on $E$ such that 
\[
\sup_{n\in\N}\|T_n\|_{\cB(E)}<\infty. 
\]
If $D$ is a dense subset of $E$ and for all $x\in D$,
\begin{equation}\label{eq:elementary}
\|T_n x-x\|_E\to 0 
\quad\mbox{as}\quad n\to\infty,
\end{equation}
then \eqref{eq:elementary} holds for all $x\in E$.
\end{lemma}
\subsection{Density of continuous function and trigonometric
polynomials in Banach function spaces}
A function of the form
\[
q(t)=\sum_{k=-n}^n\alpha_k t^k,
\quad
t\in\T,
\quad
\alpha_{-n},\dots,\alpha_n\in\C,
\]
is said to be a trigonometric (or Laurent) polynomial on $\T$. The set of 
all trigonometric polynomials is denoted by $\cP$.
\begin{lemma}\label{le:density-polynomials}
Let $X$ be a Banach function space on $\T$. The following statements
are equivalent:
\begin{enumerate}
\item[(a)]
the set $\cP$ of all trigonometric polynomials is dense in $X$;
 
\item[(b)]
the space $C$ of all continuous functions on $\T$ is dense in $X$;

\item[(c)]
the Banach function space $X$ is separable.
\end{enumerate}
\end{lemma}
\begin{proof}
The proof is developed by analogy with \cite[Lemma~1.3]{K00}.

(a) $\Rightarrow$ (b) is trivial because $\cP\subset C\subset X$.

(b) $\Rightarrow$ (c). Since $C$ is separable and $C\subset X$ is dense in 
$X$, we conclude that $X$ is separable.

(c) $\Rightarrow$ (a). Assume that $X$ is separable and $\cP$ is not dense
in $X$. Then by the corollary of the Hahn-Banach theorem (see, e.g.,
\cite[Chap.~7, Theorem~4.2]{BSU96}), there exists a nonzero functional
$\Lambda\in X^*$ such that $\Lambda(p)=0$ for all $p\in\cP$.
Since $X$ is separable, from \cite[Chap.~1, Corollaries~4.3 and~5.6]{BS88}
it follows that the Banach dual $X^*$ of $X$ is canonically isometrically
isomorphic to the associate space $X'$. Hence there exists a nonzero
function $h\in X'\subset L^1$ such that
\[
\int_\T p(t)h(t)\,dm(t)=0
\quad\mbox{for all}\quad p\in\cP.
\]
Taking $p(t)=t^n$ for $n\in\Z$, we obtain that all Fourier coefficients of
$h\in L^1$ vanish, which implies that $h=0$ a.e. on $\T$ by the uniqueness
theorem of the Fourier series (see, e.g., \cite[Chap.~I, Theorem~2.7]{Kat76}).
This contradiction proves that $\cP$ is dense in $X$.
\end{proof}
\subsection{Pointwise estimate for the Fej\'er means}
Recall that $L^1$ is a commutative Banach algebra under the convolution
multiplication defined for $f,g\in L^1$ by
\[
(f* g)(e^{i\theta})=\frac{1}{2\pi}\int_{-\pi}^\pi 
f(e^{i\theta-i\varphi})g(e^{i\varphi})\,d\varphi,
\quad
e^{i\theta}\in\T.
\]
For $n\in\N$, let
\[
K_n(e^{i\theta}):=\sum_{k=-n}^n\left(1-\frac{|k|}{n+1}\right)e^{i\theta k}
=
\frac{1}{n+1}\left(
\frac{\sin\frac{n+1}{2}\theta}{\sin\frac{\theta}{2}}
\right)^2,
\quad e^{i\theta}\in\T,
\]
be the $n$-th Fej\'er kernel. It is well-known that $\|K_n\|_{L^1}\le 1$.
For $f\in L^1$, the $n$-th Fej\'er mean of $f$ is defined
as the convolution $f*K_n$. Then
\begin{equation}\label{eq:Fejer-mean}
(f*K_n)(e^{i\theta})=
\sum_{k=-n}^n\widehat{f}(k)\left(1-\frac{|k|}{n+1}\right)e^{i\theta k},
\quad e^{i\theta}\in\T
\end{equation}
(see, e.g., \cite[Chap.~I]{Kat76}). This means that if $f\in L^1$, then
$f* K_n\in\mathcal{P}$. Moreover, if $f\in H^1=H[L^1]$, then 
$f* K_n\in\mathcal{P}_A$.
\begin{lemma}\label{le:Fejer-pointwise}
For every $f\in L^1$ and $t\in\T$,
\begin{equation}\label{eq:Fejer-pointwise-1}
\sup_{n\in\N}|(f*K_n)(t)|\le\frac{\pi^2}{2}(Mf)(t).
\end{equation}
\end{lemma}
\begin{proof}
Since $|\sin\varphi|\ge 2|\varphi|/\pi$ for $|\varphi|\le\pi/2$,
we have for $\theta\in[-\pi,\pi]$,
\begin{align}
K_n(e^{i\theta})
&\le 
\frac{\pi^2}{n+1}\frac{\sin^2\left(\frac{n+1}{2}\theta\right)}{\theta^2}
\nonumber\\
&=
\frac{\pi^2}{4}(n+1)
\frac{\sin^2\left(\frac{n+1}{2}\theta\right)}
{\left(\frac{n+1}{2}\theta\right)^2}
\nonumber\\
&\le 
\frac{\pi^2}{4}(n+1)\min\left\{1,\left(\frac{n+1}{2}\theta\right)^{-2}\right\}
\nonumber\\
&\le 
\frac{\pi^2}{2}\frac{n+1}{1+\left(\frac{n+1}{2}\theta\right)^2}
=:\Psi_n(\theta).
\label{eq:Fejer-pointwise-2}
\end{align}
It is easy to see that
\begin{equation}\label{eq:Fejer-pointwise-3}
\frac{1}{2\pi}\int_{-\pi}^\pi\Psi_n(\theta)\,d\theta\le\frac{\pi^2}{2}
\quad\mbox{for all}\quad n\in\N.
\end{equation}
From \cite[Lemma~2.11]{MS13} and estimates 
\eqref{eq:Fejer-pointwise-2}--\eqref{eq:Fejer-pointwise-3} we immediately
get estimate \eqref{eq:Fejer-pointwise-1}.
\end{proof}
\section{Proofs of the main results}\label{sec:proofs}
\subsection{Norm estimates for the Fej\'er means}
First we consider the case of rearrangement-invariant
Banach function spaces.
\begin{lemma}\label{le:uniform-boundednes-Fejer-means-RI}
Let $X$ be a rearrangement-invariant Banach function space on $\T$.
Then for each $n\in\N$, the operator $F_nf=f*K_n$ is bounded on $X$ and 
\[
\sup_{n\in\N}\|F_n\|_{\cB(X)}\le 1.
\]
\end{lemma}
\begin{proof}
By \cite[Chap.~3, Lemma~6.1]{BS88}, for every $f\in X$ and every $n\in\N$,
\[
\|f*K_n\|_X\le \|K_n\|_{L^1}\|f\|_X.
\]
It remains to recall that $\|K_n\|_{L_1}\le 1$ for all $n\in\N$.
\end{proof}
Now we will show the corresponding results for Banach function 
spaces such that the Hardy-Littlewood maximal operator is bounded on $X'$.
\begin{theorem}\label{th:uniform-boundednes-Fejer-means-BFS}
Let $X$ be a Banach function space on $\T$ such that the Hardy-Littlewood
maximal operator $M$ is bounded on its associate space $X'$. Then for each
$n\in\N$, the operator $F_nf=f*K_n$ is bounded on $X$ and
\[
\sup_{n\in\N}\|F_n\|_{\cB(X)}\le\pi^2\|M\|_{X'\to X'}.
\]
\end{theorem}
\begin{proof}
The idea of the proof is borrowed from the proof of \cite[Theorem~5.1]{CF13}.
Fix $f\in X$ and $n\in\N$. Since $K_n\ge 0$, we have $|f*K_n|\le|f|*K_n$.
Then from the Lorentz-Luxemburg theorem (see, e.g., 
\cite[Chap.~1, Theorem~2.7]{BS88}) we deduce that
\begin{align*}
\|f*K_n\|_X
&\le 
\|\,|f|*K_n\|_X=\|\,|f|*K_n\|_{X''}
\\
&=
\sup\left\{
\int_\T(|f|*K_n)(t)|g(t)|\,dm(t)\ : \ g\in X',\ \|g\|_{X'}\le 1
\right\}.
\end{align*}
Hence there exists a function $h\in X'$ such that $h\ge 0$, $\|h\|_{X'}\le 1$,
and
\begin{equation}\label{eq:uniform-boundednes-Fejer-means-BFS-1}
\|f*K_n\|_X\le 2\int_\T(|f|*K_n)(t)h(t)\,dm(t).
\end{equation}
Taking into account that $K_n(e^{i\theta})=K_n(e^{-i\theta})$
for all $\theta\in\R$, by Fubini's theorem, we get
\[
\int_\T(|f|*K_n)(t)h(t)\,dm(t)=\int_\T(h*K_n)(t)|f(t)|\,dm(t).
\]
From this identity and H\"older's inequality for $X$
(see, e.g., \cite[Chap.~1, Theorem~2.4]{BS88}), we obtain
\begin{equation}\label{eq:uniform-boundednes-Fejer-means-BFS-2}
\int_\T(|f|*K_n)(t)h(t)\,dm(t)\le\|f\|_X\|h*K_n\|_{X'}.
\end{equation}
Applying Lemma~\ref{le:Fejer-pointwise} to $h\in X'\subset L^1$, 
by the lattice property, we see that
\begin{equation}\label{eq:uniform-boundednes-Fejer-means-BFS-3}
\|h*K_n\|_{X'}\le\frac{\pi^2}{2}\|Mh\|_{X'}.
\end{equation}
Combining estimates 
\eqref{eq:uniform-boundednes-Fejer-means-BFS-1}--\eqref{eq:uniform-boundednes-Fejer-means-BFS-3}
and taking into account that $M$ is bounded on $X'$ and that $\|h\|_{X'}\le 1$,
we arrive at 
\[
\|f*K_n\|_{X}\le \pi^2\|M\|_{X'\to X'}\|f\|_X.
\]
Hence
\[
\sup_{n\in\N}\|F_n\|_{\cB(X)}
=
\sup_{n\in\N}\sup_{f\in X\setminus\{0\}}
\frac{\|f*K_n\|_X}{\|f\|_X}\le \pi^2\|M\|_{X'\to X'}<\infty,
\]
which completes the proof.
\end{proof}
\subsection{Convergence of the Fej\'er means in the norm}
The following statement is the heart of the proof of the main results.
\begin{theorem}\label{th:Fejer-norm-convergence}
Suppose $X$ is a separable Banach function space on $\T$. If $X$
is rearrangement-invariant or the Hardy-Littlewood maximal operator
is bounded on the associate space $X'$, then for every $f\in X$,
\begin{equation}\label{eq:Fejer-norm-convergence}
\lim_{n\to\infty}\|f*K_n-f\|_X=0.
\end{equation}
\end{theorem}
\begin{proof}
It is well-known that for every $f\in C$,
\[
\lim_{n\to\infty}\|f*K_n-f\|_C=0
\]
(see, e.g., \cite[Chap.~III, Theorem~1.1(a)]{C91}
or \cite[Theorem~2.11]{Kat76}).
From the definition of the Banach function space $X$ it follows that
$C\subset X\subset L^1$, where both embeddings are continuous.
Then, for every $f\in C$, \eqref{eq:Fejer-norm-convergence}
is fulfilled. From Lemma~\ref{le:density-polynomials} we know that 
the set $C$ is dense in the space $X$.
By Lemma~\ref{le:uniform-boundednes-Fejer-means-RI}
and Theorem~\ref{th:uniform-boundednes-Fejer-means-BFS},
\[
\sup_{n\in\N}\|F_n\|_{\cB(X)}<\infty,
\]
where $F_nf= f*K_n$. It remains to apply Lemma~\ref{le:elementary}.
\end{proof}
This statement for rearrangement-invariant Banach function spaces is 
contained, e.g., in \cite[p.~268]{DVL93}. Notice that the assumption of 
the separability of $X$ is hidden there.

Now we formulate the corollary of the above theorem in the case of variable
Lebesgue spaces.
\begin{corollary}
Suppose $p\in\mathfrak{P}(\T)$. If $p_+<\infty$ and $p\in LH_0(\T)$
or if $p'\in\mathfrak{K}(\T)$, then for every $f\in L^{p(\cdot)}$,
\[
\lim_{n\to\infty}\|f*K_n-f\|_{L^{p(\cdot)}}=0.
\]
\end{corollary}
For variable exponents $p\in\mathfrak{P}(\T)$ satisfying $p_+<\infty$
and $p\in LH_0(\T)$, this result was obtained by Sharapudinov
\cite[Section~3.1]{Sh96}. For $p\in\mathfrak{K}(\T)$, the above corollary
is new.
\subsection{Proofs of Theorems~\ref{th:density-analytic-polynomials-RI}
and ~\ref{th:density-analytic-polynomials-BFS}}
If $f\in H[X]$, then $p_n=f*K_n\in\cP_A$ for all $n\in\N$
in view of \eqref{eq:Fejer-mean}. By Theorem~\ref{th:Fejer-norm-convergence}, 
$\|p_n-f\|_X\to 0$ as $n\to\infty$. Thus the set $\cP_A$ is dense in 
in the abstract Hardy space $H[X]$ built upon $X$.

Moreover, if $X$ is a rearrangement-invariant Banach function space, then
from Lemma~\ref{le:uniform-boundednes-Fejer-means-RI} it follows
that $\|p_n\|_X\le\|f\|_X$ for all $n\in\N$.
\qed

\end{document}